\documentclass[pdflatex,sn-mathphys-num]{sn-jnl}

\usepackage{graphicx}%
\usepackage{multirow}%
\usepackage{amsmath,amssymb,amsfonts}%
\usepackage{amsthm}%
\usepackage{mathrsfs}%
\usepackage[title]{appendix}%
\usepackage{xcolor}%
\usepackage{textcomp}%
\usepackage{manyfoot}%
\usepackage{booktabs}%
\usepackage{algorithm}%
\usepackage{algorithmicx}%
\usepackage{algpseudocode}%
\usepackage{listings}%

\usepackage{mathtools}
\usepackage{mystyle}
\usepackage{bbm}

\usepackage[capitalize,noabbrev]{cleveref}
\crefname{equation}{eq.}{eqs.}
\usepackage{autonum}

\theoremstyle{thmstyleone}%
\newtheorem{theorem}{Theorem}
\newtheorem{lemma}[theorem]{Lemma}

\theoremstyle{thmstyletwo}%
\newtheorem{remark}{Remark}%

\theoremstyle{thmstylethree}%
\newtheorem{definition}{Definition}%

\begin{document}

\title[A $\sqrt{2}$-accelerated FISTA for composite strongly convex problems]{A $\sqrt{2}$-accelerated FISTA for composite strongly convex problems}


\author*[1]{\fnm{Kansei} \sur{Ushiyama}}\email{ushiyama.k.1053@m.isct.ac.jp}



\affil*[1]{\orgdiv{Department of Mathematical and Computing Science, School of Computing}, \orgname{Institute of Science Tokyo}, \orgaddress{\street{2-12-1 Ookayama}, \city{Meguro}, \postcode{152-8550}, \state{Tokyo}, \country{Japan}}}


\abstract{
In this paper, we propose a novel accelerated forward-backward splitting algorithm for minimizing composite convex functions expressed as the sum of a smooth function and a possibly nonsmooth function. When the composite objective is strongly convex, the proposed method achieves a linear convergence rate in objective value that improves the leading constant in the exponent by a factor of $\sqrt{2}$ relative to FISTA and also improves upon the best previously known convergence rate for this problem class. Our convergence analysis remains valid even when one of the two component functions is weakly convex, provided that their sum remains convex. The proposed algorithm is derived by discretizing a continuous-time model of the Information-Theoretic Exact Method (ITEM), an optimal first-order method for unconstrained smooth strongly convex minimization.
}



\maketitle

\section{Introduction}\label{sec1}
In this paper, we propose a new method for solving the following problem:
\begin{equation}
    \minimize_{x\in\RR^d} \quad f(x) \coloneqq g(x)+h(x),
    \label{problem}
\end{equation}
where $g\colon\RR^d\to\RR$ is differentiable and $\mu_g$-convex and has an $L_g$-Lipschitz-continuous gradient, and $h\colon\RR^d\to\RR\cup\{+\infty\}$ is proper, lower semicontinuous, and $\mu_h$-convex.
Here, a function $F$ is said to be $\lambda$-convex if $F-\frac{\lambda}{2}\norm{\cdot}^2$ is convex. Such a function is called $\lambda$-strongly convex if $\lambda>0$ and $(-\lambda)$-weakly convex if $\lambda<0$.
Let $\mu\coloneqq\mu_g+\mu_h$ and assume that $\mu\ge0$.
We further assume that a minimizer $x^\star\in\argmin_x f(x)$ exists, define the minimum value as $f^\star\coloneqq f(x^\star)$, and assume that $h$ is prox-friendly, meaning that the proximal operator
\(
    \prox_{\eta h}(x)\coloneqq\argmin_{y\in\RR^d}\setE{\eta\,h(y)+\tfrac{1}{2}\norm{y-x}^2}
\)
can be computed efficiently for any $\eta>0$ whenever it is well-defined.

When $g$ is convex (respectively, strongly convex) and $h$ is convex, problem~\eqref{problem} is known as a composite convex (respectively, composite strongly convex) optimization problem, following the terminology of~\cite{N13}. Such problems arise widely in machine learning, signal and image processing, and statistics, typically with $g$ representing a loss function and $h$ a regularizer.
A canonical example is the LASSO~\cite{T18}, in which $g$ is a least-squares loss and $h$ is an $\ell_1$-norm regularizer that promotes sparsity in the solution.
Closed convex constraints of the form $x\in C$ also fit this framework by taking $h=\iota_C$, where $\iota_C$ is the indicator function of $C$, defined by $\iota_C(x)=0$ if $x\in C$ and $\iota_C(x)=+\infty$ otherwise.

A standard approach to problem~\eqref{problem} is forward-backward splitting, which dates back to~\cite{P79} and treats $g$ and $h$ separately by taking a gradient step on $g$ followed by a proximal step on $h$.
Combining this technique with Nesterov's acceleration~\cite{N83,N18b} yields the Fast Iterative Shrinkage/Thresholding Algorithm (FISTA)~\cite{BT09}.
FISTA is widely used because of its accelerated convergence, and numerous variants have been proposed; see \cref{sec:related} for a brief overview.
Our algorithm can be viewed as a FISTA-type method and, for composite strongly convex problems, achieves what is, to the best of our knowledge, the fastest known convergence rate.

Problem~\eqref{problem} generalizes the standard composite (strongly) convex settings by allowing either $g$ or $h$ to be weakly convex while requiring the overall objective $f$ to remain convex.
Problem~\eqref{problem} can be reformulated as a standard composite (strongly) convex problem by adding $\frac{\mu_h}{2}\normm{x}$ to $g(x)$ and subtracting the same term from $h(x)$. Under this reformulation, $g$ becomes $(\mu_g + \mu_h)$-convex and $h$ becomes convex; recall that $\mu_g + \mu_h = \mu \ge 0$.
Moreover, the modified nonsmooth term $h(x) - \frac{\mu_h}{2}\normm{x}$ remains prox-friendly; see \cref{sec:experiment}.
Consequently, existing methods for composite (strongly) convex problems can be applied to~\eqref{problem} through this equivalent reformulation.

In this paper, however, we study the original formulation~\eqref{problem} directly, without reformulation, as in~\cite{B16}. Our motivation is to handle a weakly convex function $h$ directly through its proximal operator.
Weakly convex regularizers have recently been used as alternatives to the $\ell_1$-norm regularizer to obtain less biased sparse solutions in inverse problems such as sparse recovery and imaging; see \cref{sec:related}.
Accordingly, in the spirit of FISTA, we derive an algorithm that directly uses the proximal operator of $h$, thereby allowing existing library implementations to be used without reformulating the regularizer; see, for example,~\cite{CCCP}.

The derivation of our algorithm is based on discretizing the ordinary differential equation (ODE)
\begin{equation}
    \ddot{x}(t)+3\sqrt{\mu}\coth(\sqrt{\mu}t)\dot{x}(t)+2\nabla f(x(t))=0,
    \label{ITEMODE}
\end{equation}
where $\dot{x}$ and $\ddot{x}$ denote $\dv{x}{t}$ and $\frac{\dd^2x}{\dd t^2}$, respectively.
This ODE was introduced in~\cite{KY23b} as a continuous-time model of the Information-Theoretic Exact Method (ITEM)~\cite{TD23}, an optimal first-order method for smooth strongly convex minimization.
Exploiting the correspondence between this ODE and our discrete-time algorithm yields a clear and concise convergence analysis.

\subsection{Contributions}
\label{sec:cont}

\begin{itemize}
    \item In \cref{sec:ITEMODE}, we establish that ITEM ODE~\eqref{ITEMODE} for a $\mu$-convex function $f$ achieves the convergence rate
    \begin{equation}
        f(x(t))-f^\star \le
        \begin{cases}
            \min\paren*{\dfrac{12\mu}{(1-\e^{-2\sqrt{\mu}t})^2}\e^{-2\sqrt{\mu}t},\dfrac{3}{t^2}}\normm{x_0-x^\star} & \text{if } \mu>0,\\
            \dfrac{1}{t^2}\normm{x_0-x^\star} & \text{if } \mu=0.
        \end{cases}
        \label{thm2second}
    \end{equation}
    When $\mu>0$, this bound provides not only an exponential convergence rate but also a fast nonasymptotic sublinear rate. The latter can be sharper for small values of $t$, particularly when $\mu$ is small.

    \item In \cref{sec:discretization}, we discretize ITEM ODE~\eqref{ITEMODE} and construct a new accelerated forward-backward splitting algorithm for problem~\eqref{problem}, as presented in \cref{alg}. Although the algorithm requires prior knowledge of $\mu_g$, $\mu_h$, and $L_g$, it can readily be extended to the case in which $L_g$ is unknown, as discussed in \cref{sec:remarks}.

    \item In \cref{sec:convergence}, as a discrete-time counterpart of~\eqref{thm2second}, we show that the proposed algorithm achieves the convergence rate
    \begin{equation}
        f(x_k)-f^\star \le
        \begin{cases}
            \min\paren*{\dfrac{L_g-\mu_g}{2}r^{-k+1},\dfrac{6L_g}{k^2}}\normm{x_0-x^\star} & \text{if } \mu>0,\\
            \dfrac{2L_g}{k^2}\normm{x_0-x^\star} & \text{if } \mu=0,
        \end{cases}
        \label{cord21}
    \end{equation}
    where
    \[
        q_1\coloneqq\frac{\mu_g}{L_g}, \qquad q_2\coloneqq\frac{\mu_h}{L_g}, \qquad r\coloneqq\dfrac{1+q_2+\sqrt{(q_1+q_2)(2-q_1+q_2)}}{1-q_1}>1.
    \]
    When $\mu_g>0$ and $\mu_h=0$, problem~\eqref{problem} specializes to a composite strongly convex minimization problem. To the best of our knowledge, the proposed algorithm achieves the fastest convergence rate among existing accelerated methods for this setting; see \cref{tab:comparison}. In the general convex case $\mu=0$, that is, $\mu_g=-\mu_h$, the rate coincides with that of FISTA.
\end{itemize}

\begin{table}[htbp]
    \caption{Comparison of existing accelerated methods for $\mu_g>0$ and $\mu_h=0$. Let $k$ denote the iteration count and $q\coloneqq\mu_g/L_g$.}
    \label{tab:comparison}
    \centering
    \begin{tabular}{ccc}
        \toprule
        Method & Convergence rate of $f(x_k)-f^\star$ for $q \ll 1$ & Applicability\\
        \cmidrule(r){1-1}\cmidrule(l){2-3}
        Strongly convex FISTA~\cite{CP16,CC19,FV19,FV20,ADA21}
        & $\Order{\min\paren*{(1-\sqrt{q})^k,1/k^2}}$ & $\mu_g$ may be 0\\
        ADR~\cite[Thm.~4.2]{ADR22}
        & $\Order{(1+\sqrt{2q}-6q)^{-k}}$ & \\
        SR2 strongly convex FISTA (ours)
        & $\Order{\min\paren*{(1+\sqrt{2q}+q)^{-k},1/k^2}}$ & $\mu_g$ may be 0\\
        \bottomrule
    \end{tabular}
\end{table}

\subsection{Related work}
\label{sec:related}

\bmhead{Accelerated algorithms for composite convex and composite strongly convex problems}

Forward-backward splitting is an operator-splitting method for composite convex and composite strongly convex problems (cf.~\cite{RB16}) and can be traced back to~\cite{P79}.
For $\ell_1$-regularized problems, this method is known as the Iterative Shrinkage/Thresholding Algorithm (ISTA)~\cite{DDD04}.
Its iteration is given by $x_{k+1}=\prox_{\eta h}(x_k-\eta\nabla g(x_k))$ for some $\eta>0$, and its convergence rates in terms of $f(x_k)-f^\star$ are $\Order{1/k}$ when $g$ is convex and $\Order{(1-\mu/L_g)^k}$ when the objective is strongly convex.

For composite convex problems, FISTA~\cite{BT09} achieves the convergence rate $f(x_k)-f^\star\le\Order{1/k^2}$.
The fastest method for this setting is OptISTA~\cite{JGR25}, which exactly matches the corresponding lower bound.
In addition, several methods have been proposed to accelerate the convergence of the gradient-mapping norm~\cite{MS13,GL16,KF18}.
Furthermore, the authors of~\cite{LSY24} showed that a modified FISTA attains asymptotic linear convergence for strongly convex objectives.

For strongly convex objectives, several accelerated methods have been proposed~\cite{CP16,CC19,FV19,FV20,ADA21,RC22}.
They achieve the rate $f(x_k)-f^\star=\Order{\min\paren*{(1-\sqrt{\mu/L_g})^k,1/k^2}}$.
The fastest known method to date is that of~\cite{ADR22}, which improves the geometric rate of these methods by a factor of $\sqrt{2}$; see \cref{tab:comparison}.
The recent work~\cite{UTDT26}, which cited the preprint version of this paper, proposed an algorithm called Prox-ITEM. The authors showed that Prox-ITEM achieves the convergence rate $\norm{x_k-x^\star}^2=\Order{(1-\sqrt{\mu/L_g})^{2k}}$ and that this rate exactly matches the corresponding lower bound.
However, because $f(x_k)-f^\star$ cannot, in general, be upper-bounded by $\norm{x_k-x^\star}^2$ in the nonsmooth setting, the convergence rate of Prox-ITEM in terms of $f(x_k)-f^\star$ remains unclear.

These methods require the strong-convexity parameter $\mu$ as an input. By contrast, restart-based algorithms that achieve fast linear convergence without prior knowledge of $\mu$ were proposed in~\cite{LX15,ACDLR24}.

\bmhead{Weakly convex regularizers}

The $\ell_1$-norm regularizer promotes sparsity in the solution but also introduces bias because it penalizes large coefficients linearly.
To mitigate this bias, weakly convex regularizers are often used in place of the $\ell_1$-norm.
Such regularizers often become constant or nearly constant for large coefficients.
Typical examples include the Minimax Concave Penalty (MCP)~\cite{Z10} and the Smoothly Clipped Absolute Deviation (SCAD)~\cite{FL01}:

\newcommand{\MCP}[3]{\operatorname{MCP}\!\left(#1;#2,#3\right)}
\begin{equation}
\label{MCP}
\MCP{x}{\lambda}{\gamma}
=
\begin{cases}
\lambda\,|x|-\dfrac{x^2}{2\gamma}, & |x|\le\gamma\lambda,\\[6pt]
\dfrac{1}{2}\,\gamma\lambda^2, & |x|>\gamma\lambda.
\end{cases}
\qquad(\lambda>0,\ \gamma>1)
\end{equation}

\newcommand{\SCAD}[3]{\operatorname{SCAD}\!\left(#1;#2,#3\right)}
\[
\SCAD{x}{\lambda}{a}
=
\begin{cases}
\lambda\,|x|, & |x|\le\lambda,\\[6pt]
\dfrac{-x^2+2a\lambda|x|-\lambda^2}{2(a-1)}, & \lambda<|x|\le a\lambda,\\[8pt]
\dfrac{a+1}{2}\,\lambda^2, & |x|>a\lambda.
\end{cases}
\qquad(\lambda>0,\ a>2)
\]
Other examples of weakly convex regularizers can be found in, for example,~\cite{BW21,YSSMGS19}.

Problem~\eqref{problem} allows the regularizer $h$ to be weakly convex.
The convergence of ISTA for this problem was analyzed in~\cite{B16}. Moreover, although this is not stated explicitly in the literature, a variant of the FISTA-type algorithm proposed in~\cite{CP16} appears to be applicable to this setting.
Under alternative assumptions that do not require the overall objective to be convex, numerous alternative algorithms have also been proposed; see, for example,~\cite{OCBP14,GL16,LL15,TSP18,DP19,LMS21,BW21}.

\bmhead{ODE approach}

Some optimization methods can be modeled by ODEs obtained by taking the limit as their step sizes tend to zero. Such models provide intuitive insights into algorithms, facilitate clear convergence proofs, and serve as guidelines for constructing new methods.
Classical examples include the gradient flow associated with the steepest descent method and a second-order ODE describing the motion of a particle subject to friction for the heavy-ball method~\cite{P64,WRJ21}.
A recent milestone was the derivation by Su et al.~\cite{SBC16} of a second-order ODE with a vanishing damping term as a continuous-time model of Nesterov's accelerated gradient method~\cite{N83}. This model was later extended to accelerated mirror descent~\cite{KBB15}.
Since then, ODEs with similar structures have been studied extensively~\cite{AC17,M17,AD17,ACR19,APR16,ACFR222,ABCR222,ACPR18,ACR18,BCL21,ABCR22}.
Specifically, continuous-time models of ITEM and the Triple Momentum Method (TMM)~\cite{VFL18} were derived in~\cite{KY23b}.
Moreover, high-resolution ODEs, which incorporate step sizes and thereby model algorithms more accurately, have been analyzed in~\cite{SDJS22,L22,SGK20,MZY23}.

In the ODE approach, forward-backward splitting can be incorporated in two ways: by representing the splitting directly in continuous time through a proximal operator, as in~\cite{BC18,BK21,BH24}, or by viewing the splitting as a discretization strategy for dynamics that do not themselves involve a proximal operator.
In this paper, we adopt the latter strategy, as detailed in \cref{sec:discretization}.

\bmhead{Comparison of ITEM ODE with other ODEs}

The seemingly unusual damping coefficient $3\sqrt{\mu}\coth(\sqrt{\mu}t)$ in ITEM ODE~\eqref{ITEMODE} can be understood through its relationship with the continuous-time models of the Triple Momentum Method (TMM)~\cite{VFL18} and the Optimized Gradient Method (OGM)~\cite{KF16}.

TMM minimizes $\mu$-strongly convex functions with $L$-Lipschitz-continuous gradients and achieves the convergence rate $\norm{x_k-x^\star}^2\le\Order{(1-\sqrt{\mu/L})^{2k}}$, which is optimal up to a constant factor, whereas ITEM achieves the exactly optimal rate.
OGM minimizes convex functions with $L$-Lipschitz-continuous gradients and achieves the exactly optimal convergence rate $f(x_k)-\min f\le\Order{1/k^2}$.
The ITEM algorithm asymptotically reduces to TMM as $k\to\infty$ and reduces to OGM as $\mu\to0$.

The continuous-time models of these methods reflect these relationships through the damping coefficient $3\sqrt{\mu}\coth(\sqrt{\mu}t)$.
Since $3\sqrt{\mu}\coth(\sqrt{\mu}t)\to3\sqrt{\mu}$ as $t\to\infty$, this coefficient can be interpreted as a time-dependent refinement of the constant coefficient $3\sqrt{\mu}$, which is optimal in a certain sense~\cite{ADR22} and appears in the TMM ODE
\[
    \ddot{x}(t)+3\sqrt{\mu}\dot{x}(t)+2\nabla f(x(t))=0.
\]
The TMM ODE achieves the convergence rate $\norm{x(t)-x^\star}^2\le\Order{\e^{-2\sqrt{\mu}\,t}}$~\cite{ADR22}, whereas the ITEM ODE sharpens this bound to $\norm{x(t)-x^\star}^2\le\Order{\cosh^{-2}(\sqrt{\mu}\,t)}$~\cite{USM24}, as shown in~\eqref{lyapu2convv} in the subsequent analysis.

Moreover, since $3\sqrt{\mu}\coth(\sqrt{\mu}t)\to3/t$ as $\mu\to0$, the ITEM ODE reduces to the OGM ODE
\[
    \ddot{x}(t)+\frac{3}{t}\dot{x}(t)+2\nabla f(x(t))=0
\]
in the limit $\mu\to0$.
The convergence rate of the ITEM ODE also recovers that of the OGM ODE in this limit, as shown by~\eqref{lyapu2conv} in the subsequent analysis.
Therefore, the ITEM ODE smoothly interpolates between the TMM ODE and the OGM ODE, just as the ITEM algorithm interpolates between their discrete-time counterparts.

\subsection{Notation and organization of the paper}
\label{sec:notation}

\bmhead{Notation}

In this paper, $\inpr{\cdot}{\cdot}$ denotes the Euclidean inner product, and $\norm{\cdot}$ denotes the Euclidean norm.
For a time-dependent function $x\colon\RR\to\RR^d$, we use $\dot{x}$ to denote its time derivative.
For problem~\eqref{problem}, $f^\star$ denotes the optimal value of $f$, and $x^\star$ denotes an optimal solution.
For a possibly nonconvex function $h\colon\RR^d\to\RR\cup\{+\infty\}$, $\partial h(x)$ denotes the Fr\'echet subdifferential of $h$ at $x$ (cf.~\cite{RW98b}).

A function $F\colon\RR^d\to\RR$ is said to be $\mu$-convex if $F-\frac{\mu}{2}\norm{\cdot}^2$ is convex.
Such a function is called $\mu$-strongly convex if $\mu>0$ and $(-\mu)$-weakly convex if $\mu<0$.
A differentiable function $F$ is said to have an $L$-Lipschitz-continuous gradient if $\norm{\nabla F(x)-\nabla F(y)}\le L\norm{x-y}$ for all $x,y\in\RR^d$.

\bmhead{Hyperbolic functions}

The hyperbolic functions are defined by
\begin{equation}
    \sinh(u)=\frac{\e^u-\e^{-u}}{2}, ~~~
    \cosh(u)=\frac{\e^u+\e^{-u}}{2}, ~~~
    \tanh(u)=\frac{\sinh(u)}{\cosh(u)}, ~~~
    \coth(u)=\frac{\cosh(u)}{\sinh(u)},
\end{equation}
where the definition of $\coth(u)$ applies to $u\ne0$.
For any $u\in\RR$, the identities
\[
    \cosh^2(u)-\sinh^2(u)=1, \quad \sinh(2u)=2\sinh(u)\cosh(u), \quad \cosh(2u)=2\sinh^2(u)+1
\]
hold.
Furthermore, $\sinh(u)\ge u$ for all $u\ge0$, and $\abs{\tanh(u)}<1$ for all $u\in\RR$.

\bmhead{Organization}

The remainder of this paper is organized as follows.
In \cref{sec:ITEMODE}, we analyze the convergence rate of the ITEM ODE.
In \cref{sec:discretization}, we discretize the ODE and introduce a new algorithm.
In \cref{sec:convergence}, we analyze the algorithm in parallel with the analysis in \cref{sec:ITEMODE}.
Finally, in \cref{sec:remarks}, we comment on the parameters $L_g$ and $\mu$ and discuss future directions.

\section{Convergence rate of ITEM ODE}\label{sec:ITEMODE}

In this section, we present a convergence analysis of ITEM ODE~\eqref{ITEMODE}, which serves as the foundation for the derivation of our algorithm. The continuous-time analysis in this section is logically independent of the subsequent discrete-time analysis of the proposed method and is therefore not required for the latter. Nevertheless, we provide convergence proofs for ITEM ODE because the continuous- and discrete-time arguments proceed in parallel, offering intuitive insight into the otherwise intricate algebraic manipulations in the discrete-time proof.

We consider the unconstrained minimization problem $\min_{x\in\RR^d} f(x)$, where $f$ is differentiable and $\mu$-convex with $\mu\ge0$.
We examine the following first-order reformulation of ITEM ODE~\eqref{ITEMODE}:
\begin{equation}
    \simulparen{
        \dot{x}(t) &= 2\sqrt{\mu}\coth(\sqrt{\mu}t)(v(t)-x(t)),\\
        \dot{v}(t) &= \frac{\tanh(\sqrt{\mu}t)}{\sqrt{\mu}}\paren*{\mu(x(t)-v(t))-\nabla f(x(t))},
    }
    \qquad x(0)=v(0)=x_0\in\RR^d.
    \label{ITEM1}
\end{equation}
The coefficients in this ODE are interpreted by continuity when $\mu=0$. In particular,
\[
    2\sqrt{\mu}\coth(\sqrt{\mu}t)\to\frac{2}{t},
    \qquad
    \frac{\tanh(\sqrt{\mu}t)}{\sqrt{\mu}}\to t
\]
as $\mu\to0$.

Following~\cite{USM24}, we define the Lyapunov function
\begin{equation}
    E(t)=\frac{\sinh^2(\sqrt{\mu}t)}{\mu}\paren*{f(x(t))-f^\star-\frac{\mu}{2}\normm{x(t)-x^\star}}+\cosh^2(\sqrt{\mu}t)\normm{v(t)-x^\star},
    \label{lyapuc}
\end{equation}
where the expression is again interpreted by continuity when $\mu=0$.
The following lemma states that $E$ is nonincreasing. This property immediately yields
\begin{align}
    f(x(t))-f^\star-\frac{\mu}{2}\normm{x(t)-x^\star}
    &\overset{(*)}{\le} \frac{\mu E(t)}{\sinh^2(\sqrt{\mu}t)}
    \le \frac{\mu E(0)}{\sinh^2(\sqrt{\mu}t)}
    =\frac{\mu\normm{x_0-x^\star}}{\sinh^2(\sqrt{\mu}t)},
    \label{lyapu2conv}\\
    \normm{v(t)-x^\star}
    &\overset{(*)}{\le} \frac{E(t)}{\cosh^2(\sqrt{\mu}t)}
    \le \frac{E(0)}{\cosh^2(\sqrt{\mu}t)}
    =\frac{\normm{x_0-x^\star}}{\cosh^2(\sqrt{\mu}t)}.
    \label{lyapu2convv}
\end{align}
Here, the first inequalities marked by $(*)$ follow from the nonnegativity of the two terms defining $E(t)$. In particular, the $\mu$-convexity of $f$ implies $f(x(t))-f^\star-\frac{\mu}{2}\normm{x(t)-x^\star}\ge0$.

\begin{lemma}[\cite{USM24}]\label{prop:ITEMODE}
    Let $f\colon\RR^d\to\RR$ be differentiable and $\mu$-convex with $\mu\ge0$, and let $x^\star$ be a minimizer of $f$. Then, for a solution $(x,v)$ of ODE~\eqref{ITEM1}, the Lyapunov function $E$ defined by~\eqref{lyapuc} is nonincreasing. Consequently, inequalities~\eqref{lyapu2conv} and~\eqref{lyapu2convv} hold.
\end{lemma}

In \cref{sec:convergence}, we establish a discrete-time counterpart of this lemma.
To highlight the correspondence between the continuous- and discrete-time analyses, we present its proof here, although the argument is essentially the same as that in~\cite{USM24}.

\begin{proof}
    In this proof, we suppress the argument $t$ of $x(t)$ and $v(t)$.
    We show that $\dot{E}(t)\le0$.
    \begin{align}
        \MoveEqLeft \dot{E}(t)
        = \frac{2\sinh(\sqrt{\mu}t)\cosh(\sqrt{\mu}t)}{\sqrt{\mu}}\paren*{f(x)-f^\star-\frac{\mu}{2}\normm{x-x^\star}}\\
        &\quad +\frac{\sinh^2(\sqrt{\mu}t)}{\mu}\paren*{\inpr{\nabla f(x)}{\dot{x}}-\mu\inpr{x-x^\star}{\dot{x}}}\\
        &\quad +2\sqrt{\mu}\sinh(\sqrt{\mu}t)\cosh(\sqrt{\mu}t)\normm{v-x^\star}
        +2\cosh^2(\sqrt{\mu}t)\inpr{v-x^\star}{\dot{v}}\\
        &=\frac{2\sinh(\sqrt{\mu}t)\cosh(\sqrt{\mu}t)}{\sqrt{\mu}}
        \bigg(
            f(x)-f^\star-\frac{\mu}{2}\normm{x-x^\star}
            +\inpr{\nabla f(x)-\mu(x-x^\star)}{v-x}\\
        &\hspace{120pt}
            +\inpr*{v-x^\star}{\mu(v-x^\star)+\mu(x-v)-\nabla f(x)}
        \bigg)\\
        &=\frac{2\sinh(\sqrt{\mu}t)\cosh(\sqrt{\mu}t)}{\sqrt{\mu}}
        \paren*{f(x)-f^\star-\frac{\mu}{2}\normm{x-x^\star}
        +\inpr{\nabla f(x)-\mu(x-x^\star)}{x^\star-x}}\\
        &=\frac{2\sinh(\sqrt{\mu}t)\cosh(\sqrt{\mu}t)}{\sqrt{\mu}}
        \paren*{f(x)-f^\star-\inpr{\nabla f(x)}{x-x^\star}
        +\frac{\mu}{2}\normm{x-x^\star}}\\
        &\le0,
    \end{align}
    where the second equality follows from the equations for $\dot{x}$ and $\dot{v}$ in~\eqref{ITEM1}, and the last inequality follows from the $\mu$-convexity of $f$.
\end{proof}

By combining~\eqref{lyapu2conv} with a bound on $\norm{x(t)-x^\star}$, we obtain a convergence rate for $f(x(t))-f^\star$ in the following theorem. A similar technique can be found in~\cite{ADR22}.

\begin{theorem}
    Let $f$ be differentiable and $\mu$-convex with $\mu\ge0$. Then, for a solution $(x,v)$ of ODE~\eqref{ITEM1}, we have
    \begin{align}
        f(x(t))-f^\star
        &\le \paren*{1+2\tanh^2\paren*{\frac{\sqrt{\mu}}{2}t}}
        \frac{\mu}{\sinh^2(\sqrt{\mu}t)}
        \normm{x_0-x^\star}.
        \label{thm2first}
    \end{align}
    This implies the convergence rate~\eqref{thm2second}.
\end{theorem}

\begin{proof}
    When $\mu=0$, inequality~\eqref{thm2first} follows directly from~\eqref{lyapu2conv}, where $\mu/\sinh^2(\sqrt{\mu}t)$ is interpreted as $1/t^2$. Hence, we assume that $\mu>0$ in the remainder of the proof.

    Let $y(t)\coloneqq\sinh^2(\sqrt{\mu}t)(x(t)-x^\star)$.
    Then, by the first equation of~\eqref{ITEM1}, we have
    \begin{align}
        \dot{y}(t)
        &=2\sqrt{\mu}\sinh(\sqrt{\mu}t)\cosh(\sqrt{\mu}t)(x(t)-x^\star)
        +\sinh^2(\sqrt{\mu}t)\dot{x}(t)\\
        &=2\sqrt{\mu}\sinh(\sqrt{\mu}t)\cosh(\sqrt{\mu}t)
        \paren*{x(t)-x^\star+\frac{1}{2\sqrt{\mu}}\tanh(\sqrt{\mu}t)\dot{x}(t)}\\
        &=2\sqrt{\mu}\sinh(\sqrt{\mu}t)\cosh(\sqrt{\mu}t)(v(t)-x^\star).
    \end{align}
    Together with~\eqref{lyapu2convv}, this yields
    \begin{equation}
        \norm{\dot{y}(t)}
        =2\sqrt{\mu}\sinh(\sqrt{\mu}t)\cosh(\sqrt{\mu}t)\norm{v(t)-x^\star}
        \le2\sqrt{\mu}\sinh(\sqrt{\mu}t)\norm{x_0-x^\star}.
    \end{equation}
    Thus, noting that $y(0)=0$, we obtain
    \begin{align}
        \MoveEqLeft
        \norm{x(t)-x^\star}
        =\norm*{\frac{y(t)}{\sinh^2(\sqrt{\mu}t)}}
        =\frac{1}{\sinh^2(\sqrt{\mu}t)}
        \norm*{\int_0^t\dot{y}(s)\dd s}\\
        &\le\frac{1}{\sinh^2(\sqrt{\mu}t)}
        \int_0^t\norm{\dot{y}(s)}\dd s\\
        &\le\frac{2\sqrt{\mu}}{\sinh^2(\sqrt{\mu}t)}
        \int_0^t\sinh(\sqrt{\mu}s)\dd s\,\norm{x_0-x^\star}\\
        &=\frac{2(\cosh(\sqrt{\mu}t)-1)}{\sinh^2(\sqrt{\mu}t)}
        \norm{x_0-x^\star}.
    \end{align}
    Squaring this inequality and using the identities
    \[
        \sinh(\sqrt{\mu}t)=2\sinh\paren*{\frac{\sqrt{\mu}}{2}t}\cosh\paren*{\frac{\sqrt{\mu}}{2}t},
        \qquad
        \cosh(\sqrt{\mu}t)-1=2\sinh^2\paren*{\frac{\sqrt{\mu}}{2}t}
    \]
    (see \cref{sec:notation}), we obtain
    \begin{align}
        \frac{\mu}{2}\normm{x(t)-x^\star}
        &\le\frac{\mu}{2}
        \frac{\paren*{4\sinh^2\paren*{\frac{\sqrt{\mu}}{2}t}}^2}
        {\sinh^2(\sqrt{\mu}t)
        \paren*{2\sinh\paren*{\frac{\sqrt{\mu}}{2}t}
        \cosh\paren*{\frac{\sqrt{\mu}}{2}t}}^2}
        \normm{x_0-x^\star}\\
        &=\frac{2\mu\tanh^2\paren*{\frac{\sqrt{\mu}}{2}t}}
        {\sinh^2(\sqrt{\mu}t)}
        \normm{x_0-x^\star}.
    \end{align}
    Combining this inequality with~\eqref{lyapu2conv} yields~\eqref{thm2first}.

    Finally, inequality~\eqref{thm2second} follows from
    \begin{gather}
        1+2\tanh^2\paren*{\frac{\sqrt{\mu}}{2}t}\le3,
        \qquad
        \mu t^2\le\sinh^2(\sqrt{\mu}t),
        \qquad
        \frac{1}{\sinh^2(\sqrt{\mu}t)}
        =\frac{4\e^{-2\sqrt{\mu}t}}{(1-\e^{-2\sqrt{\mu}t})^2}.
    \end{gather}
    For details on these inequalities, see \cref{sec:notation}.
\end{proof}

\section{Derivation of the algorithm by discretizing ITEM ODE}\label{sec:discretization}

In this section, we discretize ITEM ODE~\eqref{ITEM1} using a weak discrete gradient (wDG)~\cite{USM23} and thereby derive our algorithm.
The discretization is designed so that the discrete-time convergence proof parallels the continuous-time proof of \cref{prop:ITEMODE}.
The proof of \cref{prop:ITEMODE} relies on three ingredients:
(i) the chain rule $\dd f(x)/\dd t=\inpr{\nabla f(x)}{\dot{x}}$,
(ii) the ODE itself, and
(iii) the $\mu$-convexity inequality involving $\nabla f$.

In the discrete-time analysis, item~(ii) is replaced by the definition of the discretized scheme.
Item~(i), however, has no exact analogue because differentiation is inherently a continuous-time operation.
Instead, we employ a weak discrete gradient $\WDG f$, which provides a discrete surrogate for the chain rule.
Moreover, $\WDG f$ is compatible with item~(iii): an analogue of the $\mu$-convexity inequality holds when $\nabla f$ is replaced by $\WDG f$.
These properties allow us to reproduce the structure of the continuous-time proof in the discrete-time setting.

The following definition extends the notion of a weak discrete gradient to nonsmooth functions.

\begin{definition}[Weak discrete gradient (cf.~\cite{USM23})]\label{def:wdg}
    For a proper convex function $f\colon\RR^d\to\RR\cup\{+\infty\}$, a \emph{weak discrete gradient} $\WDG f\colon\RR^d\times\RR^d\rightrightarrows\RR^d$ is a set-valued map for which there exist parameters $\alpha\ge0$ and $\beta,\gamma\in\RR$ satisfying $\beta+\gamma\ge0$ such that, for all $x,y,z\in\RR^d$,
    \begin{align}
        f(y)-f(x) &\le \inpr*{u}{y-x}+\frac{\alpha}{2}\normm{y-z}-\frac{\beta}{2}\normm{z-x}-\frac{\gamma}{2}\normm{y-x}
         \text{ for all } u\in\WDG f(y,z),\label{wdgsc}\\
        \WDG f(x,x) &= \partial f(x).
    \end{align}
\end{definition}

For problem~\eqref{problem}, we define
\begin{equation}
    \WDG f(y,z)\coloneqq\nabla g(z)+\partial h(y).
    \label{ITEMwdg}
\end{equation}
Before verifying that this map is a weak discrete gradient, we explain how its definition corresponds to forward-backward splitting.
As a simple example, consider the gradient flow
\[
    \dot{x}(t)\in-\nabla g(x(t))-\partial h(x(t))
\]
and its discretization using the weak discrete gradient~\eqref{ITEMwdg}:
\begin{equation}
    x_{k+1}=x_k-\eta u,
    \qquad
    u\in\WDG f(x_{k+1},x_k)=\nabla g(x_k)+\partial h(x_{k+1}),
    \label{PGwdg}
\end{equation}
where $\eta>0$ is the step size.
Whenever $1+\eta\mu_h>0$, the proximal subproblem is strongly convex, and the proximal operator is characterized by
\begin{equation}
    \xi_{k+1}=\xi_k-\eta u,\quad u\in\partial h(\xi_{k+1})
    \quad\Longleftrightarrow\quad
    \xi_{k+1}=\prox_{\eta h}(\xi_k).
    \label{im2prox}
\end{equation}
whenever the proximal map is well-defined.
Therefore, under this condition, discretization~\eqref{PGwdg} is equivalent to the standard forward-backward splitting method
\[
    x_{k+1}=\prox_{\eta h}(x_k-\eta\nabla g(x_k)).
\]
Our accelerated forward-backward splitting method is derived by applying this weak discrete gradient to ITEM ODE rather than to the gradient flow.

The next lemma shows that~\eqref{ITEMwdg} is a weak discrete gradient with parameters $(\alpha,\beta,\gamma)=(L_g,\mu_g,\mu_h)$.
The proof follows the same line of reasoning as that in~\cite{USM23}, except that we allow $\mu_g$ and $\mu_h$ to be negative, a case excluded in~\cite{USM23}.

\begin{lemma}[cf.~\cite{USM23}]\label{lem:wdg}
    Let $g\colon\RR^d\to\RR$ be differentiable and $\mu_g$-convex with an $L_g$-Lipschitz-continuous gradient, and let $h\colon\RR^d\to\RR\cup\{+\infty\}$ be proper, lower semicontinuous, and $\mu_h$-convex. Suppose that $\mu_g+\mu_h\ge0$, and let $f\coloneqq g+h$. Then, the map $\WDG f$ defined by~\eqref{ITEMwdg} is a weak discrete gradient of $f$ with parameters $(\alpha,\beta,\gamma)=(L_g,\mu_g,\mu_h)$.
\end{lemma}

\begin{proof}
    For any $x,y,z\in\RR^d$ and $\eta\in\partial h(y)$, the following inequalities hold (cf.~\cite[Lemma~2.1]{DD19} and \cite[Appendix~A]{ADA21}):
    \begin{align}
        g(y)-g(z) &\le \inpr{\nabla g(z)}{y-z}+\frac{L_g}{2}\normm{y-z},\\
        g(z)-g(x) &\le \inpr{\nabla g(z)}{z-x}-\frac{\mu_g}{2}\normm{z-x},\\
        h(y)-h(x) &\le \inpr{\eta}{y-x}-\frac{\mu_h}{2}\normm{y-x}.
    \end{align}
    Adding these inequalities yields~\eqref{wdgsc}. Moreover,
    \[
        \WDG f(x,x)=\nabla g(x)+\partial h(x)=\partial f(x),
    \]
    which proves the second condition in \cref{def:wdg}.
\end{proof}

In the following discretization and convergence analysis, we use the specific weak discrete gradient defined by~\eqref{ITEMwdg}.
However, the results remain valid for general weak discrete gradients by replacing $L_g$, $\mu_g$, $\mu_h$, and $\mu$ with $\alpha$, $\beta$, $\gamma$, and $\beta+\gamma$, respectively.

We now consider a discretization of ITEM ODE~\eqref{ITEM1}, where $x_k$ and $v_k$ approximate $x(t_k)$ and $v(t_k)$, respectively, and $z_k$ is an auxiliary variable.
The motivation for introducing $z_k$ will become clear in the subsequent convergence analysis.
We consider the scheme
\begin{equation}
    \simulparen{
        x_{k+1}-x_k &= \frac{\sinh^2(\sqrt{\mu}\,t_{k+1})-\sinh^2(\sqrt{\mu}\,t_k)}{\sinh^2(\sqrt{\mu}\,t_k)}(v_{k+1}-x_{k+1}),\\
        v_{k+1}-v_k &= \frac{\cosh^2(\sqrt{\mu}\,t_{k+1})-\cosh^2(\sqrt{\mu}\,t_k)}{2\mu\cosh^2(\sqrt{\mu}\,t_k)}
        \left(\mu_gz_k+\mu_hx_{k+1}-\mu v_{k+1}-u\right),\\
        &\hspace{100pt}u\in\WDG f(x_{k+1},z_k)=\nabla g(z_k)+\partial h(x_{k+1}),\\
        z_k-x_k &= \frac{\sinh^2(\sqrt{\mu}\,t_{k+1})-\sinh^2(\sqrt{\mu}\,t_k)}{\sinh^2(\sqrt{\mu}\,t_{k+1})}(v_k-x_k),
    }
    \label{ITEMscheme}
\end{equation}
with $t_0=0$ and $v_0=x_0\in\RR^d$.
At $k=0$, since $t_0=0$, the first equation of~\eqref{ITEMscheme} is interpreted as $x_1=v_1$.
Note that this discretization is well-defined in the limit $\mu\to0$.

The following observations show that~\eqref{ITEMscheme} can be regarded as a consistent discretization of~\eqref{ITEM1}.
\begin{itemize}
    \item As $t_{k+1}\to t_k$, we have
    \begin{align}
        \frac{x_{k+1}-x_k}{t_{k+1}-t_k}
        &\approx\frac{x(t_{k+1})-x(t_k)}{t_{k+1}-t_k}
        \to\dot{x}(t_k),\\
        \frac{\sinh^2(\sqrt{\mu}\,t_{k+1})-\sinh^2(\sqrt{\mu}\,t_k)}
        {(t_{k+1}-t_k)\sinh^2(\sqrt{\mu}\,t_k)}
        &\to\frac{2\sqrt{\mu}\sinh(\sqrt{\mu}\,t_k)\cosh(\sqrt{\mu}\,t_k)}
        {\sinh^2(\sqrt{\mu}\,t_k)}
        =2\sqrt{\mu}\coth(\sqrt{\mu}\,t_k).
    \end{align}
    Therefore, the first equation of~\eqref{ITEMscheme} discretizes
    \[
        \dot{x}(t_k)=2\sqrt{\mu}\coth(\sqrt{\mu}\,t_k)(v(t_k)-x(t_k)),
    \]
    which is the first equation of~\eqref{ITEM1} evaluated at $t=t_k$.

    \item By a similar argument, the second equation of~\eqref{ITEMscheme} discretizes
    \begin{equation}
        \dot{v}(t_k)\in\frac{\tanh(\sqrt{\mu}\,t_k)}{\sqrt{\mu}}
        \paren*{\mu_gz_k+\mu_hx(t_k)-\mu v(t_k)-\WDG f(x(t_k),z_k)}.
        \label{d2c2}
    \end{equation}

    \item The last equation of~\eqref{ITEMscheme} implies that $z_k\to x_k\approx x(t_k)$ as $t_{k+1}\to t_k$. Therefore, using $\mu=\mu_g+\mu_h$, equation~\eqref{d2c2} converges to the second equation of~\eqref{ITEM1} evaluated at $t=t_k$.
\end{itemize}

The time sequence $\setE{t_k}_{k=0}^{\infty}$ is defined recursively by choosing each $t_{k+1}$ as the largest value satisfying
\begin{equation}
    (L_g-\mu_g)(A_{k+1}-A_k)^2-2(1+\mu A_k)A_{k+1}\le0,
    \qquad
    A_k\coloneqq\frac{\sinh^2(\sqrt{\mu}\,t_k)}{\mu},
    \label{A}
\end{equation}
where when $\mu=0$, $A_k$ is interpreted by continuity as $A_k=t_k^2$.
This choice is motivated by the subsequent convergence analysis, in which we derive condition~\eqref{A} to ensure that the discrete-time Lyapunov function is nonincreasing and establish a convergence rate of $\Order{1/A_k}$.
When $L_g>\mu_g$, setting the inequality in~\eqref{A} to equality and solving for the largest admissible value of $A_{k+1}$ yields
\begin{equation}
    A_{k+1}=\frac{(L_g+\mu_h)A_k+1+\sqrt{\mu(2L_g-\mu_g + \mu_h)A_k^2+2(L_g+\mu_h)A_k+1}}{L_g-\mu_g}.
    \label{Aschedule}
\end{equation}
When $L_g=\mu_g$, $A_1$ can be chosen arbitrarily large while satisfying~\eqref{A}.
We therefore formally set $t_1=\infty$ and interpret scheme~\eqref{ITEMscheme} in the limit $t_1\to\infty$.
Then, $x_1$ is an optimal solution; see \cref{app:Lmu} for details.

The update~\eqref{ITEMscheme} with~\eqref{Aschedule} can be rewritten as the explicit algorithm presented in \cref{alg} by using the identities $\sinh^2(\sqrt{\mu}\,t_k)=\mu A_k$ and $\cosh^2(\sqrt{\mu}\,t_k)=1+\mu A_k$ and then solving~\eqref{ITEMscheme} for $x_{k+1}$ and $v_{k+1}$ via the proximal-map characterization~\eqref{im2prox}.

\begin{algorithm}[H]
    \caption{SR2 Strongly Convex FISTA}
    \label{alg}
    \begin{algorithmic}[1]
        \State \textbf{Input:} $x_0$, $L_g$, $\mu_g$, $\mu_h$, $K$
        \State Initialize $v_0\gets x_0$, $A_0\gets0$, and $\mu\gets\mu_g+\mu_h$
        \For{$k=0,1,2,\ldots,K-1$}
            \State $A_{k+1}\gets\frac{(L_g+\mu_h)A_k+1+\sqrt{(\mu_g+\mu_h)(2L_g-\mu_g+\mu_h)A_k^2+2(L_g+\mu_h)A_k+1}}{L_g-\mu_g}$ \label{algA}
            \State $B_{k+1}\gets\frac{A_{k+1}}{A_{k+1}-A_k}+\frac{\mu_gA_{k+1}+\mu_hA_k}{2(1+\mu A_k)}$ \label{line:6}
            \State $z_k\gets x_k+\frac{A_{k+1}-A_k}{A_{k+1}}(v_k-x_k)$
            \State $y_{k+1}\gets\bigl[\bigl(\frac{A_k}{A_{k+1}-A_k}+\frac{\mu A_k}{2(1+\mu A_k)}\bigr)x_k+v_k+\frac{A_{k+1}-A_k}{2(1+\mu A_k)}(\mu_gz_k-\nabla g(z_k))\bigr]/B_{k+1}$
            \State $x_{k+1}\gets\prox_{\frac{A_{k+1}-A_k}{2(1+\mu A_k)B_{k+1}}h}(y_{k+1})$ \label{line:9}
            \State $v_{k+1}\gets x_{k+1}+\frac{A_k}{A_{k+1}-A_k}(x_{k+1}-x_k)$
        \EndFor
        \State \textbf{Output:} $x_K$
    \end{algorithmic}
\end{algorithm}

Because $h$ may be $(-\mu_h)$-weakly convex when $\mu_h<0$, it is not a priori clear that the proximal operator $\prox_{\eta_{k+1}h}$ used in Line~\ref{line:9} is well-defined, where
\[
    \eta_{k+1}\coloneqq\frac{A_{k+1}-A_k}{2(1+\mu A_k)B_{k+1}}.
\]
Since $h$ is $\mu_h$-convex, the function
$\eta_{k+1}h(\cdot)+\frac{1}{2}\normm{\cdot-x}$
is $(1+\eta_{k+1}\mu_h)$-strongly convex.
Therefore, the proximal operator is well-defined and single-valued whenever
$1+\eta_{k+1}\mu_h>0$.
A direct calculation, given in \cref{app:welldef}, shows that this condition is equivalent to $2+\mu(A_k+A_{k+1})>0$, which always holds because $\mu\ge0$ and $A_k,A_{k+1}\ge0$.

\section{Convergence analysis of \cref{alg}}\label{sec:convergence}
In this section, we establish the convergence rate of \cref{alg}.

We introduce the discrete-time Lyapunov function
\begin{equation}
    E_k = \frac{\sinh^2(\sqrt{\mu}\,t_k)}{\mu}\paren*{f(x_k) - f^\star - \frac{\mu}{2}\normm{x_k-x^\star}} + \cosh^2(\sqrt{\mu}\,t_k) \normm{v_k-x^\star}. \label{lyapu}
\end{equation}
The following lemma shows that the sequence $\setE{E_k}$ is nonincreasing. This monotonicity immediately yields the following two convergence bounds:
\begin{align}
    f(x_k) - f^\star - \frac{\mu}{2}\normm{x_k - x^\star} &\le \frac{\mu\normm{x_0 - x^\star}}{\sinh^2(\sqrt{\mu}\,t_k)}, 
    \label{lyapu2convd}\\
    \normm{v_k-x^\star} &\le \frac{\normm{x_0 - x^\star}}{\cosh^2(\sqrt{\mu}\,t_k)}.\label{lyapu2convvd}
\end{align}

\begin{table}[htbp]
\caption{Abbreviations for the notation used in \cref{sec:convergence}.}
\label{table:abb}
\everymath{\displaystyle}
\begin{tabular}{cccc}
    \toprule
    \multicolumn{2}{c}{Step $k$} & \multicolumn{2}{c}{Step $k+1$} \\
    \cmidrule(r){1-2}\cmidrule(l){3-4}
    Notation & Abbreviation & Notation & Abbreviation \\
    \midrule
    $x_k$                                    & $x$        & $x_{k+1}$                                    & $x^+$        \\
    $v_k$                                    & $v$        & $v_{k+1}$                                    & $v^+$        \\
    $\sinh^2(\sqrt{\mu}\,t_k)$      & $\sinh^2$  & $\sinh^2(\sqrt{\mu}\,t_{k+1})$      & $\sinh^{2+}$ \\
    $\cosh^2(\sqrt{\mu}\,t_k)$      & $\cosh^2$  & $\cosh^2(\sqrt{\mu}\,t_{k+1})$      & $\cosh^{2+}$ \\
    $\cosh(\sqrt{\mu}\,t_k)$        & $\cosh$    & $\cosh(\sqrt{\mu}\,t_{k+1})$        & $\cosh^+$    \\
    \bottomrule
\end{tabular}
\end{table}

For readability, we simplify the notation of~\eqref{ITEMscheme} and~\eqref{lyapu} by writing $\WDG f(x_{k+1}, z_k)$ for an arbitrary element $u \in \WDG f(x_{k+1}, z_k)$ and using the shorthands summarized in \cref{table:abb}, thereby rewriting~\eqref{ITEMscheme} and~\eqref{lyapu} as
\begin{equation}
    \simulparen{
        x^+ - x &= \frac{\sinh^{2+} -\sinh^2}{\sinh^2} (v^+ - x^+),\\
        v^+ - v &= \frac{\cosh^{2+} -\cosh^2}{2\mu \cosh^2}\paren*{\mu_g z + \mu_h x^+ - \mu v^{+} - \WDG f(x^+,z)},\\
        z - x &= \frac{\sinh^{2+} - \sinh^2}{\sinh^{2+}} (v - x), \label{ITEMSchemeabb}
    }
\end{equation}
and
\begin{equation}
    E_k = \frac{\sinh^2}{\mu}\paren*{f(x) - f^\star - \frac{\mu}{2}\normm{x-x^\star}} + \cosh^2 \normm{v-x^\star},
\end{equation}
respectively.

\begin{lemma}\label{thm:convergence}
    Let $f=g+h$ be defined as in \cref{sec1}, and let $(x_{k+1},v_{k+1})$ be obtained from $(x_k,v_k)$ by the update~\eqref{ITEMscheme} at times $t_k$ and $t_{k+1}$ satisfying condition~\eqref{A}. Then, $E_{k+1}\le E_k$.
    Consequently, inequalities~\eqref{lyapu2convd} and~\eqref{lyapu2convvd} hold.
\end{lemma}
\begin{proof}
    We show that $E_{k+1} - E_k \le 0$ by analogy with the continuous counterpart $\dot{E} \le 0$ in the proof of \cref{prop:ITEMODE}.
    \begin{align}
        \MoveEqLeft E_{k+1} - E_k\\
        &= \frac{\sinh^{2+}-\sinh^2}{\mu}\paren*{f(x^+) - f^\star - \frac{\mu}{2}\normm{x^+-x^\star}} 
        \\
        &\quad + \frac{\sinh^2}{\mu} \paren*{f(x^+) - f(x) - \frac{\mu}{2}(\normm{x^+ - \xs} - \normm{x -\xs})}\\
        &\quad + (\cosh^{2+} - \cosh^2)\cdot \normm{v^+ - x^\star} + \cosh^2 \cdot\, (\normm{v^+ - \xs} - \normm{v - \xs})\\
        &\overset{(\circ)}{\le} \frac{\sinh^{2+}-\sinh^2}{\mu}\paren*{f(x^+) - f^\star - \frac{\mu}{2}\normm{x^+-x^\star}}\\
        &\quad + \frac{\sinh^2}{\mu} \left( \inpr{\WDG f(x^+,z)}{x^+ - x} + \frac{L_g}{2}\normm{x^+ - z} - \frac{\mu_g}{2}\normm{z - x} - \frac{\mu_h}{2}\normm{x^+ - x}\right.\\
        &\qquad\qquad -\left. \mu\paren*{\inpr{x^+-\xs}{x^+-x} - \frac12\normm{x^+-x}} \right)\\
        &\quad + (\cosh^{2+} - \cosh^2) \cdot \normm{v^+ - x^\star} + \cosh^2 \cdot\, \paren*{2\inpr{v^+-\xs}{v^+-v} - \normm{v^+-v}}\\
        &= \frac{\sinh^{2+}-\sinh^2}{\mu}\paren*{f(x^+) - f^\star - \frac{\mu}{2}\normm{x^+-x^\star}}\\
        &\quad + \frac{\sinh^2}{\mu}\inpr{\WDG f(x^+,z) - \mu(x^+-\xs)}{x^+ - x}\\
        &\quad + (\cosh^{2+} - \cosh^2) \cdot \inpr*{v^+ - x^\star}{v^+ - x^\star + \frac{2\cosh^2}{\cosh^{2+}-\cosh^2}(v^+-v)}\\
        &\quad + \frac{\sinh^2}{\mu}\paren*{\frac{L_g}{2}\normm{x^+ - z} - \frac{\mu_g}{2}\normm{z - x} + \frac{\mu_g}{2}\normm{x^+ - x}} - \cosh^2 \cdot\, \normm{v^+-v}\\
        &\overset{(\bullet)}{=} \frac{\sinh^{2+}-\sinh^2}{\mu}\bigg( f(x^+) - f^\star - \frac{\mu}{2}\normm{x^+-x^\star} \\
        &\qquad\qquad + \inpr{\WDG f(x^+,z) - \mu(x^+-\xs)}{v^+ - x^+}\\
        &\qquad\qquad + \inpr*{v^+ - x^\star}{\mu(v^+ - x^\star) + \mu_g z + \mu_h x^+ - \mu v^{+} - \WDG f(x^+,z)} \bigg)\\
        &\quad + \frac{\sinh^2}{\mu}\paren*{\frac{L_g}{2}\normm{x^+ - z} - \frac{\mu_g}{2}\normm{z - x} + \frac{\mu_g}{2}\normm{x^+ - x}} - \cosh^2 \cdot\, \normm{v^+-v}\\
        &= \frac{\sinh^{2+}-\sinh^2}{\mu}\left(f(x^+) - f^\star - \frac{\mu}{2}\normm{x^+-x^\star} - \inpr{\WDG f(x^+,z)}{x^+ - x^\star} \right.\\
        &\qquad\quad \left. - \mu\inpr{x^+-\xs}{v^+ - x^+} + \inpr{v^+ - x^\star}{\mu_g (z - x^\star) + \mu_h (x^+ - x^\star)} \right)\\
        &\quad + \frac{\sinh^2}{\mu}\paren*{\frac{L_g}{2}\normm{x^+ - z} - \frac{\mu_g}{2}\normm{z - x} + \frac{\mu_g}{2}\normm{x^+ - x}} - \cosh^2 \cdot\,  \normm{v^+-v}\\
        &= \frac{\sinh^{2+}-\sinh^2}{\mu}\left(f(x^+) - f^\star - \inpr{\WDG f(x^+,z)}{x^+ - x^\star} + \frac{\mu_g}{2}\normm{z-x^\star} + \frac{\mu_h}{2}\normm{x^+-x^\star} \right.\\
        &\qquad\quad \left. +\mu_g\paren*{ - \frac{1}{2}\normm{x^+-x^\star} - \inpr{x^+-\xs}{v^+ - x^+} + \inpr{v^+ - x^\star}{z - x^\star} - \frac1{2}\normm{z-x^\star}}\right)\\
        &\quad  + \frac{\sinh^2}{\mu}\paren*{\frac{L_g}{2}\normm{x^+ - z} - \frac{\mu_g}{2}\normm{z - x} + \frac{\mu_g}{2}\normm{x^+ - x}} - \cosh^2 \cdot\,  \normm{v^+-v}\\
        &\overset{(\star)}{\le} \frac{\sinh^{2+}-\sinh^2}{\mu}\paren*{\frac{L_g}{2} \normm{x^+ -z} + \frac{\mu_g}{2}\normm{v^+-x^+} - \frac{\mu_g}{2}\normm{v^+-z}}\\
        &\quad  + \frac{\sinh^2}{\mu}\paren*{\frac{L_g}{2}\normm{x^+ - z} - \frac{\mu_g}{2}\normm{z - x} + \frac{\mu_g}{2}\normm{x^+ - x}} - \cosh^2 \cdot\,  \normm{v^+-v}\\
        &\eqqcolon (\text{err}),\\
    \end{align}
    where, in the inequalities marked by $(\circ)$ and $(\star)$, we use the wDG inequality~\eqref{wdgsc} with $(\alpha,\beta,\gamma)=(L_g,\mu_g,\mu_h)$ from \cref{lem:wdg}. Specifically, we use the following two forms, respectively: 
    \begin{equation}
        f(x^+) -f(x) \le \inpr{\WDG f(x^+,z)}{x^+ - x} + \frac{L_g}{2}\normm{x^+ - z} - \frac{\mu_g}{2}\normm{z - x} - \frac{\mu_h}{2}\normm{x^+ - x}, \label{wdg1}
    \end{equation}
    which is a discrete analogue of the chain rule, and 
    \begin{equation}
        f(x^+) - f(x^\star) - \inpr{\WDG f(x^+,z)}{x^+ - x^\star} + \frac{\mu_g}{2}\normm{z-x^\star} + \frac{\mu_h}{2}\normm{x^+-x^\star} \le \frac{L_g}{2} \normm{x^+ - z}, \label{wdg2}
    \end{equation}
    which is a discrete analogue of the $\mu$-convex inequality.
    In the inequality marked by $(\star)$, we also use the identity
    \[
        \inpr{u_1 - u_3}{u_2 - u_3} = \frac12 \paren*{\normm{u_1 - u_3} + \normm{u_2 - u_3} - \normm{u_1 - u_2}} \quad \text{for} \quad u_1 ,u_2 ,u_3 \in \RR^d.
    \]
    The equality marked by $(\bullet)$ follows from applying the update~\eqref{ITEMSchemeabb} to $x^+ - x$ and $v^+ - v$.
    
    To continue the calculation of (err), we apply the following equality to the terms multiplied by $\mu_g$.
    \begin{lemma}\label{lem:betas}
    It holds that
        \begin{equation}
            \begin{split}
                \MoveEqLeft
                \frac{\sinh^{2+}-\sinh^2}{2} \paren*{\normm{v^+-x^+} - \normm{v^+-z}} + \frac{\sinh^2}{2}(-\normm{z-x} + \normm{x^+-x})\\
                &= -\frac{\sinh^{2+}}{2}\normm{x^+-z}.\label{lemma5}
            \end{split}  
        \end{equation}
    \end{lemma}
    \begin{proof}
        We use the identity \[ \normm{\lambda a + (1-\lambda) b} = \lambda\normm{a} + (1-\lambda)\normm{b} - \lambda(1-\lambda)\normm{a-b}\]
        for any $a, b \in \RR^d$ and $\lambda \in \RR$.
        Using the update~\eqref{ITEMSchemeabb}, we obtain
        \begin{align}
            \MoveEqLeft
            \normm{v^+ - z} = \normm*{\frac{\sinh^2}{\sinh^{2+} - \sinh^2}(x^+ - x) + x^+ - z}\\
            &= \paren*{\frac{\sinh^{2+}}{\sinh^{2+} - \sinh^2}}^2 \normm*{\frac{\sinh^2}{\sinh^{2+}}(x^+ - x) + \frac{\sinh^{2+} - \sinh^2}{\sinh^{2+}}(x^+ - z)}\\
            &= \paren*{\frac{\sinh^{2+}}{\sinh^{2+} - \sinh^2}}^2 \left(\frac{\sinh^2}{\sinh^{2+}}\normm{x^+ - x} + \frac{\sinh^{2+} - \sinh^2}{\sinh^{2+}}\normm{x^+ -z}\right. \\
            &\hspace{100pt} \left. - \frac{\sinh^2}{\sinh^{2+}}\frac{\sinh^{2+} - \sinh^2}{\sinh^{2+}}\normm{x-z}\right)\\
            &= \frac{\sinh^2\cdot\sinh^{2+}}{(\sinh^{2+} - \sinh^{2})^2}\normm{x^+ - x} + \frac{\sinh^{2+}}{\sinh^{2+} - \sinh^2}\normm*{x^+ - z} - \frac{\sinh^{2}}{\sinh^{2+} - \sinh^2}\normm*{x - z},
            \shortintertext{and}
            \MoveEqLeft
            \normm{v^+ - x^+} = \paren*{\frac{\sinh^2}{\sinh^{2+} - \sinh^2}}^2\normm{x^+ - x}.
        \end{align}
        Substituting these expression into the left-hand side of~\eqref{lemma5}, we obtain the desired result.
    \end{proof}
    We now resume the calculation of (err). By \cref{lem:betas},
    \begin{align}
        (\text{err})
        &= \frac{\sinh^{2+}}{\mu} \frac{L_g}{2}\normm{x^+ - z} - \frac{\mu_g}{\mu} \frac{\sinh^{2+}}{2}\normm{x^+-z} -\cosh^2 \cdot\, \normm{v^+-v}\\
        &= \sinh^{2+} \cdot\, \frac{L_g-\mu_g}{2\mu}\normm{x^+ - z}\\
        &\qquad -  \cosh^2 \cdot \paren*{\frac{\sinh^{2+}}{\sinh^{2+} - \sinh^2}}^2\normm*{x^+ - \frac{\sinh^2}{\sinh^{2+}}x - \frac{\sinh^{2+}-\sinh^2}{\sinh^{2+}}v}\\
        &= \paren*{\sinh^{2+} \cdot\,\frac{L_g-\mu_g}{2\mu} - \cosh^2 \cdot \paren*{\frac{\sinh^{2+}}{\sinh^{2+} - \sinh^2}}^2 } \normm{x^+-z},
    \end{align}
    where the last two equalities follow from the update rules for $x^+$ and $z$, respectively,
    and the last equality is precisely what motivates the update of $z$.
    The above error term is nonpositive whenever
    \begin{equation}
        \paren*{\frac{\sinh^{2+}}{\mu} - \frac{\sinh^2}{\mu}}^2\frac{L_g-\mu_g}{2} - \cosh^2\cdot\frac{\sinh^{2+}}{\mu} \le 0,
    \end{equation}
    which is equivalent to~\eqref{A} by introducing $A_k = \sinh^2(\sqrt{\mu}\,t_k) / \mu$ and using $\cosh^2(\sqrt{\mu}\,t_k) = 1 + \mu A_k$.
\end{proof}

In the following theorem, as in the continuous-time analysis, we derive a convergence rate for $f(x_k)-f^\star$ by combining inequality~\eqref{lyapu2convd} with a bound on $\norm{x_k-x^\star}$.

\begin{theorem}
    Let $f=g+h$ be defined as in \cref{sec1}, and let the sequence $\setE{t_k}$ be chosen so that $A_k=\sinh^2(\sqrt{\mu}\,t_k)/\mu$ satisfies condition~\eqref{A}. Then, the iterates $(x_k,v_k)$ generated by~\eqref{ITEMscheme} satisfy
    \begin{align}
        f(x_k)-f^\star &\le \paren*{1+2\tanh^2\paren*{\frac{\sqrt{\mu}}{2}t_k}}\frac{\mu}{\sinh^2(\sqrt{\mu}\,t_k)}\normm{x_0-x^\star} \label{thm6first}\\
        &\le
        \begin{cases}
            \dfrac{3}{A_k}\normm{x_0-x^\star} & \text{if } \mu>0,\\
            \dfrac{1}{A_k}\normm{x_0-x^\star} & \text{if } \mu=0.
        \end{cases}
    \end{align}
    In particular, \cref{alg}, or equivalently the update~\eqref{ITEMscheme} with $A_k$ determined by~\eqref{Aschedule}, achieves the convergence rate~\eqref{cord21}.
\end{theorem}
\begin{remark}
    The resulting convergence rate indicates that faster growth of $A_k$ leads to faster convergence. The fastest growth is achieved by choosing $A_k$ so that condition~\eqref{A} holds with equality, namely, by using the update~\eqref{Aschedule}. \cref{alg} adopts this update in Line~\ref{algA}.
\end{remark}
\begin{proof}
    When $\mu=0$, inequality~\eqref{thm6first} follows directly from~\eqref{lyapu2convd}.
    
    When $\mu > 0$, let $y_k \coloneqq \sinh^2(\sqrt{\mu}\,t_k) (x_k - x^\star)\, (= \sinh^2 \cdot\, (x - x^\star))$.
    Then, we have 
    \begin{align}
        y_{k+1}-y_k &= \sinh^{2+}\cdot\,(x^+ - x^\star) - \sinh^2\cdot\,(x-x^\star)\\
        &= (\sinh^{2+} - \sinh^2)\cdot(x^+- x^\star) + \sinh^2\cdot\,(x^+ - x)\\
        &= (\sinh^{2+} - \sinh^2)\cdot\paren*{x^+- x^\star + \frac{\sinh^2}{\sinh^{2+} - \sinh^2}(x^+ - x)}\\
        & \overset{(*)}{=} (\sinh^{2+} - \sinh^2)\cdot\paren*{x^+- x^\star + v^+ - x^+}\\
        &= (\sinh^{2+} - \sinh^2)\cdot(v^+- x^\star),
    \end{align}
    where in the equality marked by $(*)$, we apply the update~\eqref{ITEMSchemeabb} to $x^+ - x$.
    This, together with~\eqref{lyapu2convvd}, yields
    \begin{equation}
        \norm{y_{k+1} - y_k} = (\sinh^{2+} - \sinh^2)\cdot\norm{v^+ -x^\star} \le \frac{\sinh^{2+} - \sinh^2}{\cosh^+}\norm{x_0 - x^\star}.
    \end{equation}
    Here, the right-hand side can be bounded by using the inequality
    \[
        \frac{\sinh^{2+} - \sinh^2}{\cosh^+} \le 2(\cosh^+ - \cosh),
    \]
    which can readily be observed by
    \begin{align}
        &2(\cosh^+ - \cosh)\cosh^+ - (\sinh^{2+} - \sinh^2)\\
        = &2(\cosh^+ - \cosh)\cosh^+ - (\cosh^{2+} - \cosh^2)\\
        = &(\cosh^{+} - \cosh)^2 \ge 0.
    \end{align}
    This allows for the telescoping sum
    \begin{align}
    \MoveEqLeft
    \sum_{s=0}^{k-1} \norm{y_{s+1} - y_s} \le \sum_{s=0}^{k-1} 2\paren*{\cosh(\sqrt{\mu}\,t_{s+1}) - \cosh(\sqrt{\mu}\,t_s)} \norm{x_0 -x^\star}\\
    &= 2\paren*{\cosh(\sqrt{\mu}\,t_k) - \cosh(\sqrt{\mu}\,t_0)} \norm{x_0 -x^\star}\\
    &= 2\paren*{\cosh(\sqrt{\mu}\,t_k) - 1}\norm{x_0 -x^\star},
    \end{align}
    and, noting $y_0 = 0$, we obtain
    \begin{align}
        \MoveEqLeft 
        \norm{x_k - x^\star} = \norm*{\frac{y_k}{\sinh^2(\sqrt{\mu}\,t_k)}} = \frac{1}{\sinh^2(\sqrt{\mu}\,t_k)} \norm*{\sum_{s=0}^{k-1} (y_{s+1} - y_s) + y_0}\\
        &\le \frac{1}{\sinh^2(\sqrt{\mu}\,t_k)} \sum_{s=0}^{k-1} \norm{y_{s+1} - y_s}
        \le \frac{2(\cosh(\sqrt{\mu}\,t_k) - 1)}{\sinh^2(\sqrt{\mu}\,t_k)} \, \norm{x_0 - x^\star}.
    \end{align}
    Thus, by the same calculation as in \cref{prop:ITEMODE}, we obtain
    \[
        \frac{\mu}{2}\normm{x_k - x^\star} \le \frac{2\mu\tanh^2\paren*{\frac{\sqrt{\mu}}{2}t_k}}{\sinh^2(\sqrt{\mu}\,t_k)}\normm{x_0 - x^\star}.
    \]
    Combining this with~\eqref{lyapu2convd}, we obtain the first inequality in~\eqref{thm6first}.
    The second inequality follows from $1 + 2\tanh^2\paren*{\frac{\sqrt{\mu}}{2}t_k} \le 3$.

    To obtain the final convergence rate~\eqref{cord21}, we consider two estimates for the convergence rate $1/A_k$.
    First, by~\eqref{Aschedule}, we have
    \begin{align}
        A_{k+1} &\ge A_k + \frac{1}{L_g} + \sqrt{\frac{2}{L_g}A_k } \ge A_k + \frac{1}{2L_g} + \sqrt{\frac{2}{L_g}A_k} =  \paren*{\sqrt{A_k} + \frac{1}{\sqrt{2L_g}}}^2.
    \end{align}
    Since $A_0 = 0$, this implies the quadratic growth $A_k \ge (k/\sqrt{2L_g})^2$.
    Second, by~\eqref{Aschedule}, we have $A_1 = \frac{2}{L_g - \mu_g}$ and
    \begin{align}
        &A_{k+1} \ge \frac{L_g + \mu_h + \sqrt{\mu(2L_g - \mu_g + \mu_h)}}{L_g - \mu_g}A_k = rA_k, \\
        &\qquad \text{where} \quad r = \frac{1 + \frac{\mu_h}{L_g} + \sqrt{\paren*{\frac{\mu_g}{L_g}+\frac{\mu_h}{L_g}}\paren*{2 - \frac{\mu_g}{L_g} + \frac{\mu_h}{L_g}}}}{1-\frac{\mu_g}{L_g}} > 1.
    \end{align}
    This implies the exponential growth $A_k \ge \frac{2}{L_g - \mu_g}r^{k-1}$.
    Introducing $q_1 = \mu_g/L_g$ and $q_2 = \mu_h/L_g$, we obtain~\eqref{cord21}.
\end{proof}

In the usual composite strongly convex setting with $\mu_g > 0$ and $\mu_h = 0$,
\begin{equation}
    r = \frac{1 + \sqrt{2q_1 - q_1^2}}{1-q_1} > 1 +\sqrt{2q_1} + q_1, \label{q1rate}
\end{equation}
which is larger than the fastest previously known growth factor $1 + \sqrt{2q_1} - 6q_1 + \mathrm{o}(q_1)$~\cite{ADR22}.
The proof of the last inequality is provided in \cref{app:rate}.

\section{Remarks and discussion}\label{sec:remarks}
In this section, we provide several remarks on the parameters $L_g$ and $\mu$ and discuss possible directions for future research.

\bmhead{Adaptive selection of $L_g$}

In the proof of \cref{thm:convergence}, the Lipschitz constant $L_g$ is used only through the smoothness inequality
\[
    g(x^+) - g(z) \le \inpr{\nabla g(z)}{x^+ - z} + \frac{L_g}{2}\normm{x^+ - z}.
\]
Therefore, $L_g$ can be selected adaptively using a backtracking procedure that ensures that the above inequality holds at each iteration; see~\cite[Algorithm 19]{ADA21}.

\bmhead{Optimality of the convergence rates}

\begin{table}[htbp]
    \centering
    \caption{The fastest known convergence rates and the methods achieving them in each problem setting. Cases in which the displayed convergence rate is known to be exactly optimal are marked as \textit{optimal}.}
    \begin{tabular}{cccc}
        \toprule
        & Convex & \multicolumn{2}{c}{Strongly convex}\\
        \cmidrule{3-4}
        & $f(x_k) - f^\star$ & $f(x_k) - f^\star$ & $\norm{x_k - x^\star}^2$\\
        \midrule
        \multirow{2}{*}{Smooth} & $\Order{1/k^2}$ & $\Order{\e^{-2\sqrt{\mu/L}\,k}}$ & $\Order{\e^{-2\sqrt{\mu/L}\,k}}$\\
        & OGM, optimal & ITEM & ITEM, optimal\\
        \midrule
        \multirow{2}{*}{Composite} & $\Order{1/k^2}$ & $\Order{\e^{-\sqrt{2\mu/L}\,k}}$ & $\Order{\e^{-2\sqrt{\mu/L}\,k}}$\\
        & OptISTA, optimal & \cref{alg} & Prox-ITEM, optimal\\
        \bottomrule
    \end{tabular}
    \label{tab:rates}
\end{table}

The fastest known convergence rates and the methods achieving them in each problem setting are summarized in \cref{tab:rates}.
The Optimized Gradient Method (OGM)~\cite{KF16,DT14} and OptISTA~\cite{JGR25} achieve exactly optimal convergence rates in terms of $f(x_k) - f^\star$ for smooth convex and composite convex optimization, respectively.
ITEM and Prox-ITEM~\cite{UTDT26} attain exactly optimal convergence rates in terms of $\norm{x_k - x^\star}^2$ for smooth strongly convex and composite strongly convex optimization, respectively.
However, to the best of our knowledge, no method that is exactly optimal in terms of $f(x_k) - f^\star$ has yet been obtained even for smooth strongly convex optimization, let alone for composite strongly convex optimization.

The convergence rate of our algorithm, $f(x_k) - f^\star \le \Order{\e^{-\sqrt{2\mu/L}\,k}}$, is not directly comparable to that of Prox-ITEM, $\norm{x_k - x^\star}^2 \le \Order{\e^{-2\sqrt{\mu/L}\,k}}$.
Indeed, in the nonsmooth setting, $f(x_k) - f^\star$ cannot generally be bounded above by a constant multiple of $\norm{x_k - x^\star}^2$.
This contrasts with the smooth setting, in which a convergence rate for $\norm{x_k - x^\star}^2$ immediately yields one for $f(x_k) - f^\star$ through the inequality $f(x_k) - f^\star \le \frac{L}{2}\norm{x_k - x^\star}^2$.
Whether an algorithm achieving the rate $f(x_k) - f^\star \le \Order{\e^{-2\sqrt{\mu/L}\,k}}$ can be constructed for composite strongly convex optimization remains an open question.

\section{Numerical experiment}\label{sec:experiment}
In this section, we compare \cref{alg} with existing methods using a synthetic problem whose exact minimizer is known.
We consider problem~\eqref{problem} with dimension $d = 10000$.
Let $\mathbbm{1} \coloneqq (1,\ldots,1)^\top \in \RR^{5000}$ denote the all-ones vector, and for $u \in \RR^{5000}$, define $\normm{u}_A \coloneqq \inpr{u}{Au}$, where $A = \mathrm{diag} (1,2,\ldots,5000) \in \RR^{5000\times 5000}$. We write $x_{1:5000}$ and $x_{5001:10000}$ for the first and last 5000 components of $x \in \RR^d$. Define
\begin{equation}
    g(x) = \frac12 \normm{x_{1:5000} - 10\mathbbm{1}}_A + \frac12 \normm{x_{5001:10000} - 10^{-4}\mathbbm{1}}_A,\quad h(x) = \sum_{i=1}^d \MCP{x_i}{2}{3}, \label{exprob}
\end{equation}
as in~\eqref{MCP}.
In this setup, $L_g = 5000$, $\mu_g = 1$ and $\mu_h = -1/3$.
The unique minimizer is $x^\star = (10,\ldots,10,0,\ldots,0)^\top$, i.e., the concatenation of $10\mathbbm{1}$ and $\mathbf{0} \in \RR^{5000}$.

We compare \cref{alg} (SR2FISTA) with ISTA, the strongly convex FISTA~\cite{ADA21} (FISTA), and Algorithm~(4.17) in~\cite{ADR22} (ADR).
FISTA and ADR apply only when $h$ is convex. To handle weakly convex $h$, we adopt the following modification described in \cref{sec1}:
\[
\hat{g}(x) = g(x) + \frac{\mu_h}{2}\normm{x}, \qquad \hat{h}(x) = h(x) - \frac{\mu_h}{2}\normm{x},
\]
and run FISTA and ADR on $f = \hat{g} + \hat{h}$ with $L_{\hat{g}} = L_g + \mu_h$ and $\mu_{\hat{g}} = \mu_g + \mu_h$.
The proximal map of $\hat{h}$ is given by
\[
\prox_{\eta \hat{h}}(y) = \prox_{\frac{\eta}{1 - \mu_h\eta}h}\paren*{\frac{y}{1 - \mu_h\eta}}.
\]

The results are shown in \cref{fig:1}.
All accelerated methods (FISTA, ADR, and SR2FISTA) are faster than ISTA, and
SR2FISTA converges slightly faster than the other accelerated methods.
In the initial regime, FISTA and SR2FISTA outperform ADR, which is consistent with their $\Order{1/k^2}$ sublinear behavior in the transient phase.
The observed convergence of all the accelerated methods is faster than the reference rate $\Order{\e^{-\sqrt{2q}k}}$, where $q = \mu/(L_g + \mu_h)$, the approximate worst-case rate of ADR and SR2FISTA. While the worst-case rate of FISTA is $\Order{\e^{-\sqrt{q}k}}$, its observed convergence on this problem is similar to that of ADR and SR2FISTA.

\begin{figure}
    \centering
    \includegraphics[width=0.5\linewidth]{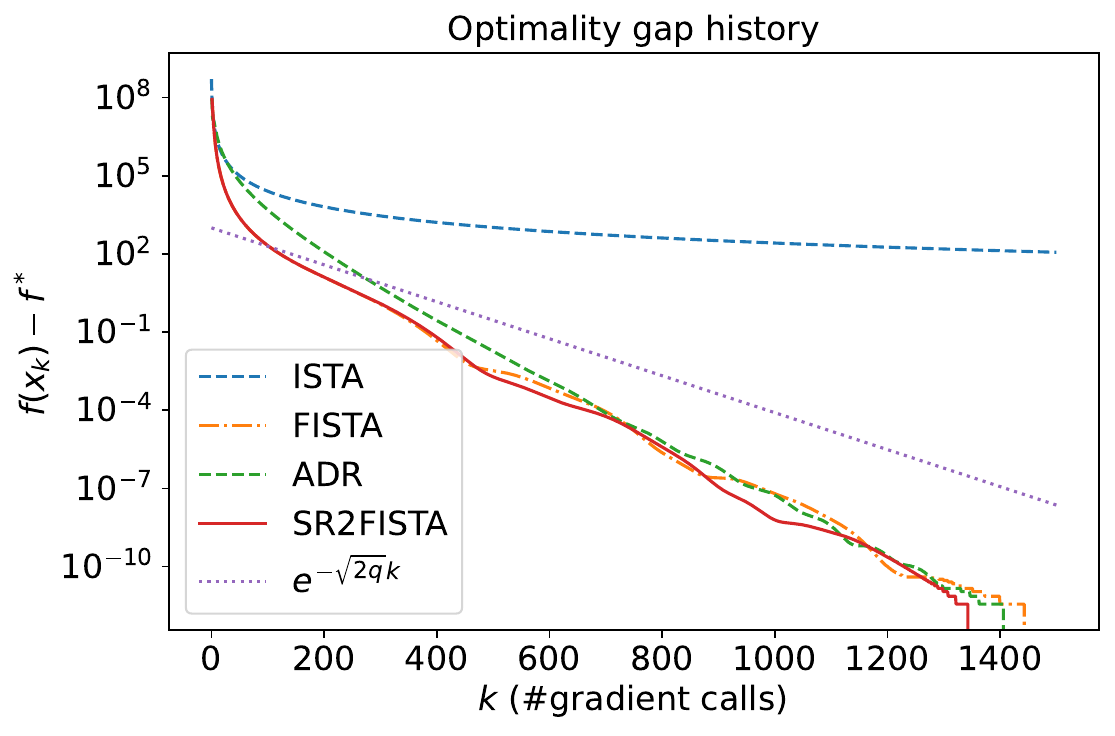}
    \caption{The convergence of $f(x_k) - f^\star$ for problem~\eqref{exprob}. The initial point is $x_0 = (1,\ldots,1) \in \RR^{10000}$. The reference line $e^{-\sqrt{2q}k}$, where $q = \mu/(L_g + \mu_h)$, is the approximate worst-case rate of ADR and SR2FISTA.}
    \label{fig:1}
\end{figure}

\backmatter









\begin{appendices}

\section{Behavior of \cref{alg} when $L_g = \mu_g$}\label{app:Lmu}
In this section, we show that when $L_g=\mu_g$, \cref{alg} returns an optimal solution in a single step by taking the limit $A_1\to\infty$.

When $L_g=\mu_g > 0$, the function $g$ is necessarily of the form $g(x)=\frac{\mu_g}{2}\norm{x-a}^2+b$ for some $a\in\RR^d$ and $b\in\RR$. Moreover, $t_1$ can be chosen arbitrarily large while satisfying condition~\eqref{A}. In the limit $t_1\to\infty$, or equivalently $A_1\to\infty$, scheme~\eqref{ITEMscheme} becomes
\begin{equation}
    \simulparen{
        0 &= v_1-x_1,\\
        0 &= \mu_g z_0+\mu_h x_1-\mu v_1-u,\quad u\in\mu_g(z_0-a)+\partial h(x_1),\\
        z_0-x_0 &= v_0-x_0.
    }
    \label{t1infty}
\end{equation}
Solving this system, we obtain
\[
    0\in\mu_g(x_1-a)+\partial h(x_1),
\]
which is precisely the optimality condition for minimizing $g+h$. Therefore, when $L_g=\mu_g$, \cref{alg} returns an optimal solution in a single step in the limit $A_1\to\infty$.

\section{Well-definedness of the proximal operator}\label{app:welldef}

In this section, we prove that the proximal operator in Line~\ref{line:9} of \cref{alg} is well-defined even when $h$ is only weakly convex.
As stated in \cref{sec:discretization}, it is sufficient to verify the condition $1+\eta_{k+1}\mu_h>0$. The following lemma provides the justification.

\begin{lemma}
    Let $\mu_g, \, \mu_h \in \RR$ satisfy $\mu \coloneqq \mu_g + \mu_h \ge 0$. For any strictly increasing sequence of nonnegative numbers $\setE{A_k}_{k\ge 0}$, define $\setE{B_k}_{k\ge 1}$ by
    \[
        B_{k+1} = \frac{A_{k+1}}{A_{k+1} -A_k} + \frac{\mu_g A_{k+1} + \mu_h A_k}{2(1 + \mu A_k)},
    \]
    as in Line~\ref{line:6} of \cref{alg}. Then, for all $k \ge 0$, we have
    \[
        \mu_h\frac{A_{k+1}-A_k}{2(1 + \mu A_k)B_{k+1}} + 1 > 0.
    \]
\end{lemma}

\begin{proof}
    Since
    \begin{align}
        \MoveEqLeft \frac{A_{k+1}-A_k}{2(1 + \mu A_k)B_{k+1}} = \frac{(A_{k+1}-A_k)^2}{2(1+\mu A_k)A_{k+1} + (A_{k+1}-A_k)(\mu_g A_{k+1} + \mu_h A_k)}\\
        &= \frac{(A_{k+1}-A_k)^2}{2A_{k+1} + \mu_g {A_{k+1}}^2 + (2\mu - \mu_g + \mu_h)A_{k+1}A_k- \mu_h {A_k}^2},
    \end{align}
    it suffices to show that
    \[
    \mu_h (A_{k+1}-A_k)^2 + 2A_{k+1} + \mu_g {A_{k+1}}^2 + (2\mu - \mu_g + \mu_h)A_{k+1}A_k- \mu_h {A_k}^2 > 0.
    \]
    The left-hand side simplifies to
    \[
        2A_{k+1} + \mu{A_{k+1}}^2 + \mu A_{k+1}A_k = A_{k+1}(2 + \mu(A_{k+1} + A_k)),
    \]
    which is positive because $A_{k+1} > 0$ and $\mu \ge 0$.
\end{proof}

\section{Proof of~\eqref{q1rate}}\label{app:rate}

Inequality~\eqref{q1rate} follows from the following lemma.

\begin{lemma}
    For $0 < q < 1$, we have
    \begin{equation}
        \frac{1 + \sqrt{2q - q^2}}{1-q} > 1 +\sqrt{2q} + q.
    \end{equation}
\end{lemma}

\begin{proof}
    Multiplying both sides by $1-q > 0$, we see that the desired inequality is equivalent to
    \begin{equation}
        1 + \sqrt{2q - q^2} > 1 +\sqrt{2q} - \sqrt{2}q\sqrt{q} - q^2.
    \end{equation}
    Subtracting $1$ from both sides and dividing by $\sqrt{q} > 0$, it suffices to show that
    \[
        \sqrt{2 - q} - \sqrt{2} + \sqrt2 q + q\sqrt{q} > 0.
    \]
    Indeed, the left-hand side satisfies
    \[
      (\text{LHS}) = - \frac{q}{\sqrt{2 - q} + \sqrt{2}} +  \sqrt2 q + q\sqrt{q} > -\frac{q}{\sqrt2} + \sqrt2 q > 0.
    \]
\end{proof}




\end{appendices}


\bmhead{Acknowledgements}
The author is grateful to Shun Sato, Akatsuki Nishioka, Keita Kume, Basil Victor P\'etusseau, and Takayasu Matsuo for their valuable comments.

\bmhead{Funding}
This work was partially supported by JSPS KAKENHI (24KJ0595). 



\section*{Statements and Declarations}
\bmhead{Competing interests} The author has no competing interests to declare that are relevant to the content of this article.


\bibliography{references}

\end{document}